\documentclass{article} 
\usepackage{amssymb,amsmath,amsthm}
\usepackage{amscd}
\usepackage{hyperref}
\usepackage{epsfig}
\usepackage{amsfonts}
\usepackage{amssymb}
\usepackage{amsmath,enumerate}
\usepackage{commath}
\usepackage{euscript}
\usepackage{enumitem}
\usepackage{amscd}
\usepackage{xcolor}
\usepackage[ruled,vlined]{algorithm2e}
\usepackage[numbers,sort&compress]{natbib}

\newtheorem{theorem}{Theorem}[section]
\newtheorem{remark}{Remark}[section]

\newtheorem{definition}{Definition}[section]

\newtheorem{lemma}{Lemma}[section]
\newtheorem{proposition}{Proposition}[section]

\makeatletter
\newsavebox{\@brx}
\newcommand{\llangle}[1][]{\savebox{\@brx}{\(\m@th{#1\langle}\)}%
	\mathopen{\copy\@brx\kern-0.5\wd\@brx\usebox{\@brx}}}
\newcommand{\rrangle}[1][]{\savebox{\@brx}{\(\m@th{#1\rangle}\)}%
	\mathclose{\copy\@brx\kern-0.5\wd\@brx\usebox{\@brx}}}
\makeatother
\begin{document}
	\title{Existence Theorems on Quasi-variational Inequalities over Banach Spaces and its Applications to Time-dependent Pure Exchange Economy}
	\author{
		Asrifa Sultana\footnotemark[1] \footnotemark[2] , Shivani Valecha\footnotemark[2] }
	\date{ }
	\maketitle
	\def\thefootnote{\fnsymbol{footnote}}
	
	\footnotetext[1]{ Corresponding author. e-mail- {\tt asrifa@iitbhilai.ac.in}}
	\noindent
	\footnotetext[2]{Department of Mathematics, Indian Institute of Technology Bhilai, Raipur - 492015, India.
	}
	
	
		
	\begin{abstract}
		We study a class of quasi-variational inequality problems defined over infinite dimensional Banach space and deduce sufficient conditions for ensuring solutions to such problems under the upper semi-continuity and pseudomonotonicity assumptions on the map defining the inequalities. The special structure of the quasi-variational inequality enables us to show the occurrence of solutions for such inequalities based on the classical existence theorem for variational inequalities. This special type of quasi-variational inequalities is motivated by the pure exchange economic problems and Radner equilibrium problems for sequential trading game. Further, we study the solvability of the specific class of quasi-variational inequalities on Banach spaces in which the constraint map may admit unbounded values. Finally, we demonstrate the occurrence of dynamic competitive equilibrium for a time-dependent pure exchange economy as an application.
	\end{abstract}
	
\section{Introduction}\label{intro}
\textcolor{black}{The notion of variational inequalities (VI) was initiated by Hartman and Stampacchia in 1966} \cite{Hartman},
concerning with partial differential equations. The VI
problem supplies us a tool to study systems of nonlinear equations \cite{Francisco-Pang,Hartman}, optimization problems \cite{Aussel_siam}, Nash equilibrium problems \cite{Aussel_2004,blum_Oettli} etc. The quasi-variational inequality (QVI) problem is indeed an extension of the VI problem where the constraint set is allowed
to rely on the variable. \textcolor{black}{We assume $X$ to be real Banach space having topological dual $X^{*}$ alongwith duality product $\langle\cdot,\,\cdot \rangle$. Suppose $D\subseteq X$ is a non-empty set.}
For given maps $T:D \rightarrow 2^{X^{*}}$ and $K:D \rightarrow 2^{D}$,
the QVI$(T,K)$ is to derive a point $\tilde{x} \in K(\tilde{x})$ so that
\begin{equation}
	~\exists\,\tilde{x}^{*} \in T(\tilde{x})~\textrm{meeting}~
	\langle \tilde{x}^{*},z-\tilde{x} \rangle \geq 0,~\textrm{for~each}~ z \in K(\tilde{x}). \label{Intro_eq1}
\end{equation}
This problem is difficult to solve computationally compared to variational inequality problem due to the fact that the constraint set is relying on the variable. Very few results for ensuring solutions for QVI$(T,K)$ are available in the literature such as the result due to Tan \cite{Tan} by assuming the upper semi-continuity property of the defined map $T$ whereas Aussel and Cotrina \cite{Aussel_existence} proved an existence result under weaker continuity and quasimonotone assumption on $T$. Further, Kien et al. \cite{Kein} derived a result for QVI$(T,K)$ by assuming that the set containing all fixed points of $K$ turns out to be closed.

It is worth to notice that every solution of QVI$(T,K)$ always a fixed point of the constraint map $K$. Due to this fact, one can think immediately to determine a solution for QVI$(T,K)$ by obtaining first all points $x$ with $x\in K(x)$ and then solve the parametric variational inequality VI$(T,K(x))$.
Aussel and Sagratella \cite{Aussel-Sagratella} have noticed that this method will not work for all QVI$(T,K)$ unless the QVI problem has a special structure.

Recently, Aussel et al. \cite{Aussel_donato_Asrifa} have studied a special type of QVI problem where some solutions can be obtained by deriving solutions of a corresponding VI problem. This special type of QVI problem is
actually motivated by the Radner equilibrium problem \cite{Aussel_donato_Asrifa,Donato_Milasi_Villanacci} constructed from sequential trading game. Indeed, this special type of  QVI problem over Banach spaces is formed as: given set-valued mappings
$F:X \rightarrow 2^{X^{*}}$ and
$K:D \rightarrow 2^{X}$, and a function $f:X \rightarrow X^{*}$, the problem QVI$(F,f,K)$ is to search for $(\tilde{d},\tilde{x}) \in D \times K(\tilde{d})$ so that there is
$\tilde{x}^{*} \in F(\tilde{x})$ satisfying
\begin{eqnarray}
	~\langle f(\tilde{x}),d-\tilde{d} \rangle+\langle \tilde{x}^{*},z-\tilde{x} \rangle  \geq 0,~\textrm{for~each}~ (d,z) \in D \times K(\tilde{d}).
	\label{def1_eq2_infinite}
\end{eqnarray}
The authors in \cite{Aussel_donato_Asrifa} derived existence results for such QVI $(\ref{def1_eq2_infinite})$ problem over finite dimensional spaces under weaker semi continuity and quasimonotone assumption on $F$. As an application, they have derived the occurrence of Radner equilibrium \cite{Aussel_donato_Asrifa} of certain sequential trading game. Few more works in this direction can be seen in Donato et al. \cite{milasicon,Donato_Milasi,Donato_Milasi_Villanacci}.

In this article we have considered this special type of quasi-variational inequalities over infinite dimensional Banach space and provide sufficient conditions for the occurrence of solutions for such QVI$(F,f,K)$ $(\ref{def1_eq2_infinite})$ problem on Banach space.
Because of the special structure of the considered quasi-variational inequalities, we first establish the presence of a solution for such QVI $(\ref{def1_eq2_infinite})$ by applying a theorem on existence of solutions for variational inequalities. In this case, the result is proved under the upper semi-continuity and pseudomonotonicity assumption on $F$. Further, considering the fact that QVI$(F,f,K)$ $(\ref{def1_eq2_infinite})$ is indeed a quasi-variational inequality, we have provided an alternative existence result for such problems where the map $F$ is assumed to be convex valued rather than $F$ is pseudomonotone. Finally, an existence result for QVI$(F,f,K)$ $(\ref{def1_eq2_infinite})$ problem having an unbounded constraint map is derived. \textcolor{black}{As an application, we have demonstrated the occurrence of dynamic competitive equilibrium for a time-dependent pure exchange economy at the end of this article.}

\section{Preliminaries and notations}\label{Prel}

We now call back few definitions, symbols and properties of the
variational notions which will be required for our main results in the upcoming section.

We assume $X$ to be real Banach space having topological dual $X^{*}$ and duality pairing $\langle\cdot,\,\cdot \rangle$. \textcolor{black}{For $D \subset X$, we use notations conv($D$) and $\overline {\textrm{conv}}(D)$ to indicate the convex hull and closed convex hull, respectively.} For $T:D \rightarrow 2^{X^{*}}$, the collection of all solutions $S(T,D)$ for the \textcolor{black}{VI problem} introduced by Stampacchia is given as,
$$\textcolor{black}{\{\tilde{x} \in D: \text{there is}\, \tilde{x}^{*} \in T(\tilde{x})~\textrm{satisfying}~\langle \tilde{x}^{*},y-\tilde{x} \rangle \geq 0,~\textrm{for~any}~ y \in D\},}$$
whereas $M(T,D)$ is symbolized for the collection of all solutions of the VI problem introduced by Minty and it is described as
$$\{\tilde{x} \in D: \langle y^{*},y-\tilde{x} \rangle \geq 0,~ \textcolor{black}{\textrm{for~each}}~y~\text{in}~D ~\textrm{and}~ y^{*}~\text{in}~T(y)\}.$$
One can check that \textcolor{black}{if the set $D$ is convex closed}, then the set $M(T,D)$ also becomes convex as well as closed.

\textcolor{black}{We call back that $T:D \rightarrow 2^{X^{*}}$ is known as \textit{pseudomonotone} map on the set $D$ if for each $x,y \in D$ and some $x^{*}$ in $T(x)$ satisfying $\langle x^{*},y-x\rangle \geq 0$ we have,
	$\langle y^{*},y-x \rangle \geq 0,~\textrm{for~all}~ y^{*}$ in $T(y).$}

\textcolor{black}{The reader may find in \cite{Aubin}, the concept of upper semi-continuous (u.s.c.) set-valued maps. Below-mentioned} proposition concerning with upper semi-continuity will be used in the sequel.
\textcolor{black}{\begin{proposition}\label{subnet_prop1}\cite[Lemma 1.10]{ansari}
		Consider $X$ is a Banach space and $T:X \rightarrow 2^{X^{*}}$ ($X^{*}$ being topological dual of $X$) is a non-empty compact valued mapping. Then, $T$ is u.s.c. at a point $y \in X$ iff for any sequence $y_n \rightarrow y$ and $y_n^{*} \in T(y_n)$, a subsequence $(y_{n_m}^{*})_m$ of $(y_n^{*})_n$ is obtained in such a way that $y_{n_m}^{*}\rightarrow y^*$ for some $y^*\in T(y)$.
\end{proposition}}

\textcolor{black}{One can revise the concepts related to lower semi-continuity and closeness of set-valued map from \cite{Aubin}.
	The below-mentioned lemma will have important contribution in the sequel as it provides the suitable conditions under which the restricted map becomes lower semi-continuous (l.s.c.).
	\begin{lemma}\label{lemma1_infin}\cite{cot2}
		Consider $X$ is a Banach space and $D\subset X$ is non-empty convex. Let $K:D \rightarrow 2^{X}$ be a convex-valued l.s.c. mapping. Define a map $K_{r} :D \rightarrow 2^{X}$ as
		$K_{r}(d) = K(d)\cap \bar{B}(0,r)$ for some $r>0$. Then the map $K_{r}$ is l.s.c. if  $K(d)\cap B(0,r)\neq \emptyset$ for any $d \in D$.
\end{lemma}}

The upcoming theorem for quasi-variational inequalities obtained by Tan \cite{Tan} will be subsequently required for ensuring the presence of solutions for QVI$(G,f,K)$ $(\ref{def1_eq2_infinite})$.
\textcolor{black}{\begin{theorem}\cite{Tan}\label{Thm_Tan}
		Consider a nonempty convex compact set $D$ in a locally convex topological vector space $X$. Assume $K:D\rightarrow 2^{D}$ and $T:D \rightarrow 2^{X^{*}}$ (considering $X^{*}$ as topological dual of $X$) are mappings meeting the conditions:
		\begin{enumerate}
			\item[(i)] $K$ is non-empty convex compact valued l.s.c. mapping with closed graph;
			\item [(ii)]$T$ fulfills the upper semi-continuity property with $T(x)$ being non-empty convex compact for each $x$ in $D$.
		\end{enumerate}
		Then QVI\,$(T,K)$ $(\ref{Intro_eq1})$ contains a solution.
\end{theorem}}
\textcolor{black}{A real-valued function $\psi$ defined on a Banach space $X$ is known as quasi-concave \cite{Aussel_siam} if for every pair $x,y\in X$ and $s\in [0,1]$,
	$$ \psi(sx+(1-s)y)\geq \min \{\psi(x),\psi(y)\}.$$
	It is worth noting that the adjusted normal operators are required in place of classical derivatives and sub-differential operators for obtaining solutions to optimization problems through variational inequalities in the case of a non-smooth quasi-concave function (for e.g. see \cite{Aussel_chapter,Aussel_siam}).} 

\textcolor{black}{Let us recall the concept of adjusted normal operator $N^a_{-\psi}$ associated to a quasi-concave function $\psi$. Any function $\psi:X\rightarrow \mathbb{R}$ is quasi-concave iff sublevel set of the corresponding quasi-convex function $-\psi$, defined as, $S_{-\psi}(x)=\{z\in X : (-\psi)(z)\leq (-\psi)(x)\}$ is convex for every $x$ in $X$ (see \cite{Aussel_chapter}). Suppose that $S^a_{-\psi}(x)$ indicates an adjusted sublevel set corresponding to a quasi-convex function $-\psi$ for any $x\in X$, which is formed according to \cite[Definition 2.3]{Aussel_siam}. Then, we recollect the notion of adjusted normal operator $N^a_{-\psi}:X\rightarrow 2^{X^*}$ from \cite{Aussel_siam},
	$$ N^a_{-\psi}(x)= \{x^*\in X^*:\langle x^*, z-x\rangle \leq 0,~\forall\, z\in S^a_{-\psi}(x)\}.$$
	As studied in \cite{Aussel_siam}, these operators satisfy generalized monotonicity and some specific continuity properties simultaneously.}

\textcolor{black}{\section{Existence Theorems on Quasi-variational Inequality Problems}\label{mainresult}
	\subsection{Case of Bounded Constraint Maps}}\label{exist}

Our goal is to \textcolor{black}{give the} sufficient conditions \textcolor{black}{which ensure the} occurrence of solutions for the considered special quasi-variational \textcolor{black}{inequality problems} QVI$(F,f,K)$ $(\ref{def1_eq2_infinite})$ over Banach spaces. The central tool is \textcolor{black}{to construct a VI problem associated to this} special QVI$(F,f,K)$ $(\ref{def1_eq2_infinite})$ problem. Then we aim to deduce the occurrence of solution for such QVI$(F,f,K)$ by \textcolor{black}{employing a preliminary existence theorem on VI problems} depending on the observation that any vector solving the corresponding VI problem \textcolor{black}{yields a solution of} QVI$(F,f,K)$ $(\ref{def1_eq2_infinite})$ problem. In this case, the result is proved under the upper semi-continuity and pseudomonotonicity assumption on $F$.

\begin{theorem}\label{Thm1_infinte_new}
	Let $D$ be a non-empty convex
	compact subset of a Banach space $X$ and $f:X \rightarrow X^{*}$ be a continuous affine function. Assume $F:X \rightarrow 2^{X^{*}}$ and
	$K:D \rightarrow 2^{X}$ are mappings meeting the conditions:
	\begin{enumerate}
		\item[(i)]$K$ is non-empty convex valued lower semicontinuous map with closed graph;
		\item [(ii)]$F$ is non-empty compact valued map satisfying the upper semi-continuity and pseudomonotonicity property on the set \textrm{conv}$(K(D))$.
	\end{enumerate}
	Therefore the QVI\,$(F,f,K)$ $(\ref{def1_eq2_infinite})$ contains a solution whenever the set $K(D)$ becomes compact.
\end{theorem}
\begin{proof}{}
	We first construct a map $H:D\rightarrow 2^{X^{*}}$ by $H=f\circ \,\Gamma$ where
	$\Gamma:D \rightarrow 2^{X}$ is defined by
	$$\Gamma(d)=S(F,K(d))~~ \textrm{for~any}~~ d \in D.$$ We observe that the solution set for QVI$(F,f,K)$ $(\ref{def1_eq2_infinite})$ is nonempty whenever a solution for the Stampacchia variational inequality $S(H,D)$ exists. Indeed, we assume that $\tilde{d} \in D$ solves $S(H,D)$. Thus an element $\tilde{x} \in \Gamma(\tilde{d})$ exists satisfying
	\begin{equation}
		\langle f(\tilde{x}),d-\tilde{d} \rangle \geq 0\quad~\textrm{for~each}~ d \in D.\label{thm1_eq2}
	\end{equation}
	Now, $\tilde{x} \in \Gamma(\tilde{d})$ indicates that  $\exists \, \tilde{x}^{*} \in F(\tilde{x})$~meeting the condition
	\begin{equation}
		\langle \tilde{x}^{*},z-\tilde{x} \rangle  \geq 0\quad\textrm{for~all}~ z\in K(\tilde{d}). \label{thm1_eq3}
	\end{equation}
	Therefore the couple $(\tilde{d},\tilde{x})$ is a solution for
	QVI$(F,f,K)$ $(\ref{def1_eq2_infinite})$.
	
	Hence, it is enough to show the occurrence of a solution for $S(H,D)$.
	It appears from \cite[Lemma 3.2]{Aussel_stabilty} and \cite[Lemma 3.1]{Aussel_stabilty} that
	$$M(F,K(d))=M(\textrm{conv}F,K(d))\subseteq S(\textrm{conv}F,K(d))=S(F,K(d)),~\textrm{for}~d \in D,$$
	where $ \textrm{conv}F:X \rightarrow 2^{X^{*}}$ is constructed as $\textrm{conv}F(x)=\textrm{conv}(F(x))$ for $x \in X$. Further, it occurs $S(F,K(d)) \subseteq M(F,K(d))$ as $F$ is pseudomonotone on the set $K(d)$. Hence we have for any $d \in D$,
	\begin{equation}
		\Gamma(d)=S(F,K(d))=M(F,K(d)). \label{eq 1_infinite}
	\end{equation}
	Since $K(d)$ is compact, the lemmas \cite[Lemma 2.1]{Aussel_2004} and
	\cite[Proposition 2.1]{Aussel_2004} jointly deduce that $M(F,K(d))\neq \emptyset$ and hence $\Gamma(d)\neq \emptyset$.
	Moreover, we get from (\ref{eq 1_infinite}) that $\Gamma(d)$ becomes convex closed set and thus $\Gamma$ becomes a non-empty closed convex valued map.
	
	Now, our aim is to show that $\Gamma$ satisfies the upper semi-continuity property. As the map $\Gamma$ is contained in a compact set $K(D)$, it is enough to show that the map $\Gamma$ is having closed graph. Let $(d_n)_{n\geq 1}$ in $D$ and $y_n \in \Gamma(d_n)$ be two sequences in order that $d_n \rightarrow d$ and $y_n \rightarrow y$. We aim to establish $y \in \Gamma(d)$. Since $y_n \in \Gamma(d_n)$ for each $n \in \mathbb{N}$, there is  $(y_n^{*})_n$ in $X^{*}$ with $y_n^{*}\in F(y_n)$ and
	\begin{equation}
		\langle y_n^{*},z-y_n \rangle  \geq 0,~\forall\, z\in K(d_n). \label{thm1_eq 2_infinite}
	\end{equation}
	Due to the fact that $F$ is \textcolor{black}{u.s.c.} at $y$, Proposition \ref{subnet_prop1} indicates that there exists $(y_{n_k}^{*})_k \subseteq (y_n^{*})_n$ with $y_{n_k}^{*} \rightarrow \textcolor{black}{y^*}$ for some $y^{*} \in F(y)$. Let $w$ be an element in $K(d)$. Since $K$ is lower semi-continuous on $D$ and $d_{n_k} \rightarrow d$, we get $w_{n_k}\in K(d_{n_k})$ having  $w_{n_k} \rightarrow w$. Hence it appears from (\ref{thm1_eq 2_infinite}),
	$$\langle y_{n_k}^{*},w_{n_k}-y_{n_k} \rangle  \geq 0.$$
	This yields by taking $k \rightarrow \infty$, $\langle y^{*},w-y \rangle  \geq 0$. As $w$ is an arbitrary element in $K(d)$, we can conclude that
	$\langle y^{*},w-y \rangle  \geq 0,$~for any $w\in K(d)$. Hence $\Gamma$ is a closed graph map. Therefore $\Gamma$ satisfies the upper semi-continuity property.
	
	The \textcolor{black}{mapping} $H=f\circ \Gamma$ is indeed \textcolor{black}{u.s.c.} on $D$ considering that $\Gamma$ is  \textcolor{black}{u.s.c.} and the function $f$ is continuous. Considering that $\Gamma(d)$ is convex for any $d \in D$ and $f$ is affine, we can easily verify that the set $H(d)$ becomes convex. Moreover, it turns out that $H(d)=f(\Gamma(d))$ is compact for any $d \in D$ as $f$ is continuous and $\Gamma(d)$ is compact. Hence $H$ \textcolor{black}{becomes an u.s.c. map with non-empty convex compact values}. \textcolor{black}{Consequently,} the variational inequality $S(H,D)$ consists one solution according to Theorem \ref{Thm_Tan} by considering the constant constraint map $K(x)=D$. Hence the proof follows.
	
\end{proof}

In the preceding result, the presence of solutions for QVI$(F,f,K)$ $(\ref{def1_eq2_infinite})$ have been shown through a result on existence of solutions for VI problem. Now, considering the fact that  QVI$(F,f,K)$ $(\ref{def1_eq2_infinite})$ is indeed a QVI problem, we now wish is to establish another result for QVI$(F,f,K)$ $(\ref{def1_eq2_infinite})$ over Banach space by applying \textcolor{black}{a preliminary} existence \textcolor{black}{result for QVI} problems. In fact, we require to presume $F$ to be convex valued rather than the pseudomonotonicity of $F$ as considered in Theorem \ref{Thm1_infinte_new}.


\begin{theorem}\label{Theorem3}
	Let $D$ be a non-empty convex
	compact subset of a Banach space $X$ and $f:X \rightarrow X^{*}$ be a continuous function. Assume
	$F:X \rightarrow 2^{X^{*}}$ and $K:D \rightarrow 2^{X}$ are mappings meeting the conditions:
	\begin{enumerate}
		\item[(i)]$K$ is non-empty convex valued lower semicontinuous map with closed graph;
		\item [(ii)] $F$ admits non-empty convex compact values and fulfills the upper semi-continuity property on \textrm{conv}$(K(D))$.
	\end{enumerate}
	Therefore the QVI\,$(F,f,K)$ $(\ref{def1_eq2_infinite})$ contains a solution whenever the set $K(D)$ becomes compact.
\end{theorem}

\begin{proof}{}
\textcolor{black}{We set $\tilde{D}=D \times\overline{\textrm{conv}}(K(D))$.} Since $K(D)$ is compact, $\overline{\textrm{conv}}(K(D))$ is also compact and hence $\tilde{D}$ is a compact convex set.  Let us construct a set-valued map $T$ from the set $\tilde{D}$ to
$X^{*} \times X^{*}$ by $T(d,x)=\{(f(x),x^{*}): x^{*} \in F(x)\}$. Suppose that $(d,x)$ is an element in $D \times \overline{\textrm{conv}}(K(D))$. Clearly, $T(d,x)$ is non-empty convex and compact. Let $V_1 \times V_2$ be an open set in $X^{*} \times X^{*}$ containing $T(d,x)=f(x) \times F(x)$. As $V_1$ is an open set containing $f(x)$ and $f$ is continuous, $f^{-1}(V_1)$ is open  in $X$ containing $x$. Due to the fact that $F$ has the upper semi-continuity property and $V_2$ is an open set with $F(x)\subseteq V_2$, there is a neighbourhood $U'$ of $x$ with $F(U') \subseteq V_{2}$. Then  $ U= U' \cap f^{-1}(V_1) $ is an open set containing $x$. Let $M$ be an open set with $d \in M$. Then $(d,x) \in M\times U$ is an open set with $T(M\times U) \subseteq V_1 \times V_2$. Hence $T$ is upper semi-continuous on $D \times \overline{\textrm{conv}}(K(D))$.

Further, we define a map $\tilde{K}:\tilde{D}=D \times \overline{\textrm{conv}}(K(D)) \rightarrow 2^{\tilde{D}}$
as $\tilde{K}(d,x)=D \times K(d)$. Due to the fact $K$ is having closed graph and $D$ is a closed set, it is easy to check $\tilde{K}$ is having closed graph. We assert that $\tilde{K}$ meets the lower semi-continuity property considering that $K$ is \textcolor{black}{a l.s.c. map}.
Indeed, the map $\tilde{K}$ is non-empty, convex and compact valued considering that $D$ is convex compact and $K(d)$ is non-empty convex compact.

Hence QVI$(T,\tilde{K})$ accommodates a solution according to Theorem \ref{Thm_Tan} and hence we can find $(\tilde{d},\tilde{x}) \in \tilde{K}(\tilde{d},\tilde{x})=D \times K(\tilde{d})$ and $(f(\tilde{x}),\tilde{x}^{*}) \in T(\tilde{d},\tilde{x})=\{(f(\tilde{x}),\tilde{x}^{*}):\tilde{x}^{*} \in F(\tilde{x})\}$ satisfying
$$\langle (f(\tilde{x}), \tilde{x}^{*}),(d,z)-(\tilde{d},\tilde{x}) \rangle\geq 0,
~~\textrm{for~every}~ (d,z) ~\textrm{lies~in}~ D \times K(\tilde{d}).$$

Thus we can conclude that the QVI$(F,f,K)$ $(\ref{def1_eq2_infinite})$ has a solution.

\end{proof}

\subsection{Case of Unbounded Constraint Maps}
\textcolor{black}{The majority of well known existence theorem on VI problems} require the compactness
of the constraint set $K$ (\cite{Francisco-Pang}). In the year 2004, Aussel and Hadjisavvas \cite{Aussel_2004} derived the presence of a solution
for variational inequalities having unbounded constraint set with the help of \textcolor{black}{below-mentioned} coercivity \textcolor{black}{criterion} $(\bar{C})$ for the
map $T:K \rightarrow 2^{X^{*}}$:
\begin{eqnarray}
(\bar{C})&\quad& \exists \,r>0,~\textrm{for~all}~ x \in K\setminus \bar{B}(0,r),~ \exists\, y \in K~\textrm{satisfying}
~ \nonumber\\
&\quad& ||y||<||x||~\textrm{and}~\langle x^{*},x-y \rangle \geq 0, ~\textrm{for~any}~ x^{*} \in T(x),\nonumber
\end{eqnarray}
\textcolor{black}{where $\bar B(0,r)=\{z\in X: ||z||\leq r\}$}. Several coercivity conditions for variational inequalities having unbounded constraint set were considered and compared by Bianchi et al. \cite{BHS}. We adapted the coercivity condition $(\bar{C})$ to
the connection of quasi-variational inequality problems in this \textcolor{black}{article}.

The coercivity \textcolor{black}{criterion $(\bar{C}_{d})$ which will be considered} in this section in connection with the quasi-variational inequalities QVI$(F,f,K)$ $(\ref{def1_eq2_infinite})$ is constructed as: for the set-valued maps $K:D\rightarrow 2^X$ and $F:X \rightarrow 2^{X^{*}}$, the coercivity \textcolor{black}{criterion $(\bar{C}_{d})$ is said to be fulfilled} at $d \in D$ when
\begin{eqnarray}
(\bar{C}_{d})&\quad& \exists \, r_{d}>0,~\textrm{for~all}~ x \in K(d)\setminus \bar{B}(0,r_{d}), \exists \, y \in K(d)~\textrm{satisfying}
\nonumber\\
&\quad& ||y||<||x||~\textrm{and}~\langle x^{*},x-y \rangle \geq 0,~ \textrm{for~any}~ x^{*} \in F(x).\nonumber
\end{eqnarray}
It is worth to observe that the both the  coercivity conditions $(\bar{C}_{d})$ and $(\bar{C})$ match to each other if constraint map $K$ considered as constant map. Moreover, the condition $(\bar{C}_{d})$ automatically satisfied if $K$ is bounded map.


\begin{theorem}\label{Thm1_infinte}
Let $D$ be a non-empty convex
compact subset of a Banach space $X$ and $f:X \rightarrow X^{*}$ be a continuous affine function. Assume $F:X \rightarrow 2^{X^{*}}$ and
$K:D \rightarrow 2^{X}$ are mappings meeting the conditions:
\begin{enumerate}
	\item[(i)]$K$ is non-empty convex valued lower semicontinuous map with closed graph;
	\item [(ii)]$F$ is non-empty convex compact valued map satisfying the upper semi-continuity and pseudomonotonicity property on $\textrm{conv}(K(D))$.
\end{enumerate}
Then the QVI\,$(F,f,K)$ $(\ref{def1_eq2_infinite})$ contains a solution if the coercivity \textcolor{black}{criterion} $(\bar{C}_{d})$ is fulfilled for each $d \in D$ and there exists $r>\sup\{r_{d}: d \in D\}$ such that the set $K(D) \cap \bar{B}(0,r)$ is compact and $K(d) \cap \bar{B}(0,r)$ is non-empty for all $d \in D$.
\end{theorem}
\begin{proof}{}
\textcolor{black}{Consider a map $K_{r}:D\rightarrow 2^{X}$ defined} as $K_{r}(d)=K(d)\cap \bar{B}(0,r)$ for $d \in D$. Because of $K(d) \cap \bar{B}(0,r)$ is non-empty, the coercivity condition $(\bar{C}_{d})$ indicates that we can find an element $z \in K(d)\cap B(0,r)$. Hence $K_{r}$ satisfies the lower semi-continuity property according to Lemma \ref{lemma1_infin}. Hence it is easy to observe that $K_{r}$ is closed, \textcolor{black}{l.s.c.} with non-empty convex values.

Define a map $H_{r}:D\rightarrow 2^{X^{*}}$ by $H_{r}=f\circ \Gamma_{r}$ where
$\Gamma_{r}:D \rightarrow 2^{X}$ is defined by
$$\Gamma_{r}(d)=S(F,K_{r}(d))~~ \textrm{for~any}~~ d \in D.$$ By \textcolor{black}{employing the same argument as} given in the proof of Theorem \ref{Thm1_infinte_new}, it occurs for any $d \in D$,
\begin{equation}
	\Gamma_{r}(d)=S(F,K_{r}(d))=M(F,K_{r}(d)). \label{eq 1_infinite}
\end{equation}
Since $K_{r}(d)=K(d)\cap \bar{B}(0,r)$ is compact, the lemmas \cite[Lemma 2.1]{Aussel_2004} and
\cite[Proposition 2.1]{Aussel_2004} jointly deduce that $M(F,K_{r}(d))\neq \emptyset$ and hence $\Gamma_{r}(d)\neq \emptyset$.
Moreover, we get from (\ref{eq 1_infinite}) that $\Gamma_{r}(d)$ becomes convex closed set and thus $\Gamma_{r}$ is eventually a non-empty closed convex valued map. \textcolor{black}{Following the proof of} Theorem \ref{Thm1_infinte_new}, it can be shown that the map $\Gamma_{r}$ is having a closed graph. Therefore it is \textcolor{black}{u.s.c.} as it is contained in the set $K(D)\cap \bar{B}(0,r)$ which is compact.

The \textcolor{black}{mapping} $H_{r}=f\circ \Gamma_{r}$ is indeed \textcolor{black}{u.s.c.} on $D$ considering that $\Gamma_{r}$ is  \textcolor{black}{u.s.c.} and the function $f$ is continuous. Considering that $\Gamma_{r}(d)$ is convex for any $d \in D$ and $f$ is affine, we can easily verify that the set $H_{r}(d)$ becomes convex. Moreover, it turns out that $H_{r}(d)=f(\Gamma_{r}(d))$ is compact for any $d \in D$ as $f$ is continuous and $\Gamma_{r}(d)$ is compact. Hence $H_{r}$ is a non-empty convex compact valued map satisfying the upper semi-continuity property.
Thus the variational inequality $S(H_{r},D)$ consists one solution according to Theorem \ref{Thm_Tan} by considering the constant constraint map $K(x)=D$. This \textcolor{black}{implies} that there exist $\tilde{d} \in D$ and $\tilde{x} \in \Gamma_{r}(\tilde{d})$
\textcolor{black}{satisfying},
\begin{equation}
	\langle f(\tilde{x}),d-\tilde{d} \rangle \geq 0,~\forall\, d \in D.\label{thm1_infit_eq2}
\end{equation}
Now $\tilde{x} \in \Gamma_{r}(\tilde{d})$ implies that $\tilde{x} \in K(\tilde{d})\cap \bar{B}(0,r)$ and \textcolor{black}{we obtain $\tilde{x}^{*}$ in $F(\tilde{x})$~ satisfying,
	\begin{equation}
		\langle \tilde{x}^{*},x-\tilde{x} \rangle  \geq 0,~\text{for each}~ x\in K(\tilde{d})\cap \bar{B}(0,r). \label{thm1_infit_eq3}
\end{equation}}
Now, if we are able to show that the above inequality (\ref{thm1_infit_eq3}) holds for any $z \in K(\tilde{d})$, then the pair $(\tilde{d},\tilde{x})$ solves the QVI $(\ref{def1_eq2_infinite})$. By using the coercivity condition $(\bar{C}_{\tilde{d}})$, \textcolor{black}{one can check that there is some} $\,\tilde{y} \in K(\tilde{d})\cap B(0,r)$ satisfying $\langle x^{*},\tilde{x} -\tilde{y}\rangle  \geq 0$,~for \textcolor{black}{every} $x^{*}\in F(\tilde{x})$. This together with (\ref{thm1_infit_eq3}) \textcolor{black}{shows that},
\begin{equation}
	\langle \tilde{x}^{*},\tilde{x} -\tilde{y}\rangle=0. \label{thm1_infit_eq4}
\end{equation}
Since $z \in K(\tilde{d})$ and $\tilde{y} \in K(\tilde{d})\cap B(0,r)$, the point $t\tilde{y}+(1-t)z \in K(\tilde{d})\cap \bar{B}(0,r)$ for some $t \in (0,1)$. Hence, we have from (\ref{thm1_infit_eq3}) and (\ref{thm1_infit_eq4}) that $\langle \tilde{x}^{*},z-\tilde{x} \rangle  \geq 0$. Since $z \in K(\tilde{d})$ is arbitrary, the inequality (\ref{thm1_infit_eq3}) holds for any $z \in K(\tilde{d})$. This completes proof.

\end{proof}
\section{Application to Dynamic Competitive Equilibrium Problem}
We will employ the existence theorems proved in Section \ref{mainresult} to show the occurrence of dynamic competitive equilibrium for time-dependent pure exchange economy defined over the time interval $[0,T], T>0$ (one may refer \cite{anello,cotrina,milasicon,milasi} for more information). For this purpose, we first characterize the dynamic competitive equilibrium problem as a QVI which is a specific case of the general problem QVI$(F,f,K)$ (\ref{def1_eq2_infinite}). Then, we demonstrate the occurrence of dynamic competitive equilibrium by employing Theorem \ref{Thm1_infinte_new} and Theorem \ref{Theorem3}. In order to deal with the dynamic competitive equilibrium problem involving a certain number of agents competing for different commodities over the time period $[0,T]$, we need the Lebesgue space $L^2([0,T],\mathbb{R}^m)$ along with the inner product $\llangle .,.\rrangle_m$ defined as, 
\begin{equation*}
\llangle h,g\rrangle_m= \int_{0}^{T} \langle h(t),g(t)\rangle_m\, dt,
\end{equation*}
where $\langle .,.\rangle_m$ represents Euclidean inner product. Further, the inner product $\llangle .,.\rrangle_m$ induces a norm on $L^2([0,T],\mathbb{R}^m)$, which we denote by $||.||_{L^2}$.

Suppose that $\mathcal J=\{1,2,\cdots, m\}$ is the set of goods and $\mathcal I=\{1,2,\cdots,n\}$ is the set of agents. Suppose $L=L^2([0,T],\mathbb{R}^m)$ and the ordered cone in $L$ is indicated as,
$$C_L=\{\alpha \in L : \alpha^j(t)~\text{is non-negative}~\text{for almost all}~t~\text{in}~[0,T],\,\forall \, j\in \mathcal J\}.$$ Let us denote the endowment and consumption plan of any agent $i$ by $e_i$ and $x_i$ in $C_L$, respectively. Considering $p_j(t)$ to be the price of the $j^{th}$ commodity at any instant $t$, we assume that the price vector $p=(p_1,p_2,\cdots p_m)\in L$ lies in the set $P$ defined as,
\begin{equation}
P=\bigg\{p\in C_L : \frac{1}{T} \int_{0}^{T} \sum_{j=1}^{m} p^j(t)dt=1\bigg\}.
\end{equation}

In order to obtain more preferable goods, the agents involved in this game compete with each other by storing, buying and selling goods over a certain time period. Suppose that the preference of agent $i$ relative to the consumption vector $x_i$ at any instant $t\in[0,T]$ is expressed through function $u_i:[0,T]\times \mathbb{R}^m\rightarrow \mathbb{R}$. Then, each agent intends to maximize his utility function $u_i$ not only at a particular instant $t$, but over the whole time period $[0,T]$. In particular, each agent $i$ wants to maximize the mean value utility function $\mathcal{U}_i:L^2([0,T],\mathbb{R}^m)\rightarrow \mathbb{R}$ given as,
\begin{equation}\label{U_i}
\mathcal{U}_i(x_i)=\int_{0}^{T} u_i(t,x_i(t)) dt.
\end{equation} 
Furthermore for any price vector $p\in P$, the expenditure of any agent $i$ cannot exceed his wealth due to endowments. Hence, the following budget constraint is applied:
\begin{align}\label{Mi}
M_i(p)=\big \{x_i\in C_L: \,\llangle p,x_i-e_i\rrangle_m \leq 0\big \}.
\end{align} 

The dynamic competitive equilibrium for time-dependent pure exchange economy is formally defined as follows  \cite{anello,milasicon,milasi}.
\begin{definition} \cite{anello,milasicon}
A price vector $\bar p\in P$ along with the consumption vector $\bar x\in \prod_{i\in\mathcal{I}} M_i(\bar p)$ forms a dynamic competitive equilibrium $(\bar p,\bar x)$ if,
\begin{align}
	&~~~\mathcal{U}_i(\bar x_i) = \max_{x_i\in M_i(\bar p)}\mathcal{U}_i(x_i)~\forall\,i\in \mathcal{I}, ~\text{and}\label{eqdef1}\\
	&\int_{0}^{T} \bigg [\sum_{i\in \mathcal{I}} (\bar x_i^{j} (t)-e_i^{j}(t))\bigg ] ~dt \leq 0~\forall\,j\in \mathcal{J}\label{eqdef2}.
\end{align}
\end{definition}

In the existing literature, the existence of dynamic competitive equilibrium is established through quasi-variational inequalities by assuming that the mean value utility functions satisfy either concavity \cite{anello,milasicon} or semi-strict quasi-concavity condition \cite{milasi}. We assume that one of the following set of assumptions (A1) or (A2) holds true for each $i\in \mathcal{I}$,
\begin{itemize}
\item[(A1)] 
\begin{itemize}
	\item[(a)] ${u}_i(t,w)$ is measurable w.r.t. $t$ for all $w\in \mathbb{R}^m$;
	\item[(b)] $u_i(t,w)$ is concave w.r.t. $w$ for almost all $t\in [0,T]$;
	\item[(c)] $u_i(t,.)\in C^1(\mathbb{R}^m_+)$ for almost all $t\in [0,T]$;
	\item[(d)] for any $j\in \mathcal J$, the function $\frac{\partial u_i}{\partial x_i^j}$ is measurable w.r.t. $t$ and $\nabla u_i$ meets the below-mentioned growth condition (see \cite{anello,rogers}),
	\begin{equation*}
		||\nabla u_i(t,w)||_{\mathbb{R}^m} \leq C_i ||w||_{\mathbb{R}^m} + g(t)~\forall\, w\in \mathbb{R}^m_+,~\text{for almost all}~t\in [0,T],
	\end{equation*}
	considering that $C_i>0$ and $g\in L^2([0,T],\mathbb{R}_+)$.
\end{itemize}
\item[(A2)] $\mathcal{U}_i:L^2([0,T],\mathbb{R}^m)\rightarrow \mathbb{R}$ defined as (\ref{U_i}) is quasi-concave and continuous.
\end{itemize}

Now, we consider a compact set $\tilde D\subset P$ (see \cite{brezis,milasicon}), 
\begin{align*}
\tilde D= \bigg \{&p\in L : p^j(t) \geq 0~\text{for all}~j\in\mathcal{J}, \sum_{j=1}^{m} p^j(t)=1 ~\text{for almost all}~t~\text{in}~[0,T],\\& p(t)=0~\text{if}~t\notin [0,T],~\lim_{u\rightarrow 0}||p(t+u)-p(t)||_{L^2}=0~\text{uniformly in}~p \bigg \}.
\end{align*}
The compactness of above set facilitates us to employ Theorem \ref{Thm1_infinte_new} and Theorem \ref{Theorem3} for deriving the occurrence of dynamic competitive equilibrium.

It can be observed that the constraint map $M_i$ defined in (\ref{Mi}) admits unbounded values. For the sake of mathematical ease, we consider a bounded map $\tilde K_i: \tilde D\rightarrow 2^L$ defined as, 
\begin{equation}
\tilde K_i(p)=M_i(p) \cap H \label{eqKi}
\end{equation} 
where, $H=\prod_{j\in \mathcal{J}} H_j$ and each set bounded convex set $H_j$ is defined as,
$$ H_j=\bigg\{\alpha\in L^2([0,T],\mathbb{R}) : \alpha(t)\geq 0~\text{for almost all}~t\in [0,T], \int_{0}^{T}\alpha(t)dt\leq r_j\bigg\}$$ 
by considering $r_j>\displaystyle \int_{0}^{T}\bigg(\sum_{i\in \mathcal{I}} e^j_i(t)\bigg)dt$.

Now, we characterize the above-stated time-dependent pure exchange economy as a quasi-variational inequality QVI (\ref{qvi}) given below:
\begin{align}
&~\text{To determine}~ (\bar p, \bar x)\in \tilde D\times \tilde K(\bar p)~\text{so that,}~\exists\,(\bar x^*_i)_{i\in \mathcal{I}}\in F(\bar x)~\text{fulfilling},\notag\\
&\llangle \sum_{i\in \mathcal{I}}(e_i-\bar x_i), p-\bar p \rrangle_{m} +\sum_{i\in \mathcal{I}}\llangle \bar x^*_i, y_i -\bar x_i\rrangle_m \geq 0,~\text{for all}~\,(p,y)\in \tilde D\times \tilde K(\bar p), \label{qvi}
\end{align}
where $\tilde K(p)=\prod_{i\in\mathcal{I}} \tilde K_i(p)$ is a bounded set as considered in (\ref{eqKi}) and the map $F:L^2([0,T], \mathbb{R}^{nm})\rightarrow 2^{L^2([0,T], \mathbb{R}^{nm})}$ is formed as:
\begin{equation} \label{F}
F(x)= \prod_{i\in \mathcal{I}} F_i(x_i)=
\begin{cases} \big{\{}(-\nabla u_1(x_1),\cdots,-\nabla u_n(x_n))\big{\}}, &\text{if (A1) holds}\\ 
	(N^a_{-\mathcal U_1}(x_1)\setminus \{0\},\cdots, N^a_{-\mathcal U_n}(x_n)\setminus \{0\}), &\text{if (A2) holds.}
\end{cases}
\end{equation}
In equation (\ref{F}), the symbol $\nabla u_i(x_i)$ denotes the function $\nabla u_i(.,x_i(.))$ and the symbol $N^a_{-\mathcal U_i}$ denotes the adjusted normal operator corresponding to the quasi-convex function $-\mathcal U_i$ (see Section \ref{Prel}).
\begin{lemma}\label{equivalence}
Consider that the conditions given in (A1) or (A2) hold true. If the pair $(\bar p, \bar x)$ solves the QVI (\ref{qvi}), then it is dynamic competitive equilibrium.
\end{lemma}
\begin{proof}{}
Any vector $(\bar p, \bar x)$ solves QVI (\ref{qvi}) iff, both the conditions (\ref{eq15}) and (\ref{eq14}) hold, that is, $\bar x_i\in \tilde K_i(\bar p)$ is a solution for the following VI problem for each $i\in \mathcal{I}$:
\begin{equation}
	\exists\, \bar x_i^*\in F_i(\bar x_i)~\text{satisfying}~\llangle \bar x^*_i, y_i -\bar x_i\rrangle_m \geq 0,\enspace \text{for any}~ y_i\in \tilde K_i(\bar p),\label{eq15}
\end{equation} 
and $\bar p \in \tilde D$ is a solution for the following VI problem:
\begin{equation}
	\llangle \sum_{i\in \mathcal{I}}(e_i-\bar x_i), p-\bar p\rrangle_{m}\geq 0\enspace \text{for any}~ p\in \tilde D.\label{eq14}
\end{equation} 
We assert that the vector $\bar x$, which solves QVI (\ref{qvi}), also fulfills (\ref{eqdef2}). By using (\ref{eq14}) and the fact $\bar x\in M(\bar p)$, we obtain following inequality for any $p\in \tilde D$, 
$$\llangle \sum_{i\in \mathcal{I}}(e_i-\bar x_i), p\rrangle_{m}=\llangle \sum_{i\in \mathcal{I}}(e_i-\bar x_i),p-\bar p\rrangle_m+\llangle \sum_{i\in \mathcal{I}}(e_i-\bar x_i), \bar p\rrangle_m\geq 0.$$ 
In the above inequality, suppose $p=(p^k)_{k\in \mathcal{J}}\in P$ is considered as $p^j=1$ for some $j$ and $p^k=0~\text{if}~k\neq j$. Then, clearly $p\in \tilde D$ and we affirm that (\ref{eqdef2}) is true.

Now, we demonstrate that any solution of (\ref{eq15}) also solves (\ref{eqdef1}). For this, we first show that the vector $\bar x_i$ is a solution of the below-stated problem for each $i\in \mathcal{I}$, 
\begin{equation}
	\llangle \bar x^*_i, y_i -\bar x_i\rrangle_m \geq 0,\enspace \text{for each}~ y_i\in M_i(\bar p).\label{eqM}
\end{equation}
Suppose $\bar x_i$ is not a solution for (\ref{eqM}) for some $i\in \mathcal{I}$. Then,
\begin{equation}
	\exists\, y'_i\in M_i(\bar p)\setminus H~\text{s.t.}~ \llangle x^*_i, y'_i-\bar x_i\rrangle_m<0,~ \text{for all}~ \, x^*_i\in F_i(\bar x_i). \label{contr}
\end{equation} 
We obtain $\epsilon >0$ satisfying $B(\bar x_i, \epsilon)\cap C_L \subseteq H$, since we have following relation due to (\ref{eqdef2}),
$$\int_{0}^{T} \bar x_i^j(t)~dt\leq \int_{0}^{T} \sum_{i\in \mathcal{I}}\bar x_i^j(t)~dt \leq \int_{0}^{T} \sum_{i\in \mathcal{I}} e_i^j(t)~dt < r_j.$$
Assume that $0<\lambda <\frac{\epsilon}{||y'_i-\bar x_i||_{L^2}}$, then $y^\circ_i=\lambda y'_i+(1-\lambda) \bar x_i \in \tilde K_i(\bar p)$. Thus, by (\ref{eq15}) we obtain $\llangle \bar x^*_i, y'_i -\bar x_i\rrangle_m \geq 0$ which contradicts our hypothesis in (\ref{contr}).

Clearly, $M_i$ admits non-empty convex closed values for any $p\in \tilde D$. Suppose assumption (A1) holds, then $F_i(x_i)= \{-\nabla u_i(x_i)\}$ and one can verify that any solution of $(\ref{eqM})$ solves the maximization problem given in (\ref{eqdef1}) (see \cite{milasiconcave,milasicon}). Further, if (A2) holds, then $F_i(x_i)= (N^a_{-\mathcal U_i}(x_i)\setminus \{0\})$ and (\ref{eqdef1}) can be deduced from (\ref{eqM}) by following \cite[Proposition 5.16]{Aussel_chapter}.

\end{proof}
Let us prove the occurrence of a dynamic competitive equilibrium by employing Lemma \ref{equivalence} and the existence results established in Section \ref{mainresult}. 
\begin{theorem}\label{application}
Suppose that one of the following hypothesis hold:
\begin{itemize}
	\item[(a)] $u_i$ satisfies assumption (A1) for every $i\in \mathcal{I}$;
	\item[(b)] $\mathcal{U}_i$ satisfies assumption (A2) for every $i\in \mathcal{I}$ s.t. the map $F$ considered as $F(x)=\prod_{i\in \mathcal{I}} (N^a_{\mathcal{-U}_i}(x)\setminus\{0\})$ admits non-empty convex compact values and fulfills the upper semi-continuity property on conv$(\tilde K(\tilde D))$.
\end{itemize}
Then, the pure exchange economy admits a dynamic competitive equilibrium whenever $\tilde K_i$ is l.s.c. map over $\tilde D$ and the set $\tilde K_i(\tilde D)$ is compact for any $i\in \mathcal{I}$. 
\end{theorem}
\begin{proof}{}
According to Lemma \ref{equivalence}, we can ensure the existence of dynamic competitive equilibrium by showing that QVI (\ref{qvi}) admits a solution.
One can notice that QVI (\ref{qvi}) is a particular instance of QVI$(F,f,K)$ (\ref{def1_eq2_infinite}). In fact, we consider $X=X^*=L^2([0,T],\mathbb{R}^{nm})$ and $D=\tilde D\times\{0\}\cdots \{0\}$ is subset of the considered Banach space $X$. Let us construct $f:X\rightarrow X^*$ as 
$f(x)= \big(\sum_{i\in \mathcal{I}} (e_i-x_i),0,\cdots,0\big )$, the map $K:D\rightarrow 2^X$ is formed as $K(d)=\prod_{i\in\mathcal{I}} \tilde K_i(p)$ (from (\ref{eqKi})) for any $d=(p,0\cdots 0)\in D$ and $F:X\rightarrow 2^{X^*}$ is considered as,
\begin{equation*} 
	F(x)=
	\begin{cases} \{(-\nabla u_1(x_1),\cdots,-\nabla u_n(x_n))\}, &\text{if hypothesis (a) holds}\\ 
		(N^a_{-\mathcal U_1}(x_1)\setminus \{0\},\cdots, N^a_{-\mathcal U_n}(x_n)\setminus \{0\}), &\text{if hypothesis (b) holds.}
	\end{cases}
\end{equation*}

Let us observe that $D\subset L$ is a non-empty convex compact set and $f$ is a continuous affine function. 
Clearly, $M_i$ is a closed map (see \cite[Lemma 3.4 (ii)]{anello}) which admits non-empty convex values. Further, one can verify that $K$ is l.s.c. map due to our hypothesis on the map $\tilde K_i$. Hence, $K$ becomes a closed lower semi-continuous map which admits non-empty convex values.

Assume that the hypothesis (a) holds, then we can verify that the map $F$ considered as $F(x)= \{\prod_{i\in \mathcal{I}} (-\nabla u_i(x_i))\}$ is continuous \cite[Theorem 10.58]{rogers} and monotone. Consequently, 
the assumptions in Theorem \ref{Thm1_infinte_new} are fulfilled if hypothesis (a) holds and QVI (\ref{qvi}) admits a solution which is a dynamic competitive equilibrium. 

Again, if hypothesis (b) holds, then we can check that all the assumptions in Theorem \ref{Theorem3} are fulfilled. Hence, we ensure the presence of a dynamic competitive equilibrium.

\end{proof}

\begin{remark}
\begin{itemize}
	\item[(a)] The assumption that $\tilde K_i$ fulfills lower semi-continuity property for any $i\in \mathcal I$ is not too restrictive. In fact, one can verify that $M_i$ is l.s.c. over $\tilde D$ if the below-stated survivability condition holds for each $j\in \mathcal{J}$ (see \cite[Lemma 3.4]{anello}),
	$$ e_i^j(t) >0,~\text{for almost all}~t\in [0,T].$$ Eventually, $\tilde K_i$ becomes l.s.c. by virtue of \cite[Lemma 2.3]{cot2}.
	\item[(b)] If the function $\mathcal U_i$ is continuous and semi-strictly quasi-concave with no maximum in $L$, that is, for any $x_i\in L$ there exists $y_i\in L$ s.t. $\mathcal{U}_i(y_i)>\mathcal{U}_i(x_i)~\forall\,i\in \mathcal{I}$, then the Walras' law is fulfilled. In fact, for any $p\in P$~ if~$(\bar x_i)_{i\in\mathcal{I}}$~is a vector s.t.~$\bar x_i\in M_i(p)$ maximizes $\mathcal{U}_i$ then clearly the condition, $\sum_{i\in \mathcal{I}}\llangle p,\bar x_i-e_i\rrangle_m =0$ holds (refer to \cite{anello,cotrina,milasiconcave}).
\end{itemize}
\end{remark}

\section*{Acknowledgement(s)}
The first author acknowledges SERB, India for providing the financial assistance under $\big(\rm{MTR/2021/000164}\big)$. The second author offers her gratitude towards UGC, India for the financial aid given by them throughout this research work under $\big(\rm{1313/(CSIRNETJUNE2019)}\big)$.

\end{document}